\documentclass[10pt]{article}
\usepackage[margin=1in]{geometry}
\usepackage{amsmath, amssymb, amsthm}
\usepackage{mathtools}
\usepackage{authblk}
\usepackage{graphicx}
\usepackage{float}
\usepackage{algorithm}
\usepackage{algpseudocode}
\usepackage{booktabs}
\usepackage{csquotes}
\usepackage{xcolor}
\usepackage[english]{babel}

\usepackage[backend=biber,style=apa,citestyle=authoryear]{biblatex}
\usepackage[colorlinks=true,linkcolor=blue,citecolor=blue,urlcolor=blue]{hyperref}

\newtheorem{theorem}{Theorem}
\newtheorem{prop}{Proposition}

\newtheorem{cor}{Corollary}
\theoremstyle{definition}

\title{Markov processes on a circular lattice}
\author{Sourav Majumdar \\\texttt{souravm@iitk.ac.in}}
\affil{Department of Management Sciences, Indian Institute of Technology Kanpur, India}
\date{}

\begin{document}
\maketitle

\abstract{We develop a Markov process viewpoint for discrete circular distributions motivated by directional-statistics settings where angles are observed on a finite grid and evolve over time. On the $m$-point discrete circle, the cycle graph, we study diffusion-generated families, obtaining an explicit transition kernel, exact trigonometric moments, and convergence to uniformity. We present a simple approach to construct reversible nearest-neighbour chains with any prescribed strictly positive stationary pmf $\pi$, providing discrete analogues of Markov processes on the continuous circle. We construct processes whose stationary laws are the discrete von Mises and wrapped Cauchy distributions with closed-form normalizers and exact moments.}


\section{Introduction}

Directional statistics (\cite{mardia2000directional,jammalamadaka2001topics}) concerns random directions and orientations, with circular data as the basic case.
In the continuous setting, a useful organizing principle is that many classical circular laws arise naturally
from stochastic processes on $\mathbb{S}^1$: Brownian motion on the circle is ergodic with the uniform law as
its equilibrium, while the von Mises distribution can be characterized as the stationary law of a
time-reversible mean-reverting diffusion on the circle, the von Mises process,  see
\cite{kent1975disc,kent1978time}), which plays the role of an Ornstein-Uhlenbeck process on $\mathbb{S}^1$.
Such process constructions are especially compelling when angles are observed sequentially, rather
than only a static (time-independent) marginal distribution; see also recent related diffusion-based developments in directional statistics such as
\cite{garcia2019langevin,majumdar2024diffusion,GarcaPortugus2025}.

Discrete circular data arise whenever directions are recorded on a finite grid:
finite-resolution sensors, discretized bearings, binned phase measurements, or discretized directional
preferences. A comprehensive recent treatment of static modeling on a circular lattice is given by
\cite{mardia2023families}, who organize families of discrete circular distributions via general construction
principles and illustrate them with applications including settings where circular observations are naturally
recorded on a finite set of directions, and where time variation may be present, such as roulette wheel and acrophase data. These examples motivate moving beyond purely static pmfs:
when observations are time-indexed, one often wants a model for how a discrete direction evolves over time,
not only what its marginal distribution is.

The goal of this paper is to develop a Markov process viewpoint for discrete circular distributions on the
simplest discrete circle: the $m$-point lattice identified with the cycle graph. Concretely, we model a
time-indexed discrete direction as an angle-valued process $\Theta_t=\theta_{X_t}$ where
$(X_t)_{t\ge 0}$ is a continuous-time Markov chain on the cycle $\mathbb{Z}_m$. This provides discrete analogues of
two canonical continuous circle constructions:
\begin{enumerate}
\item \textbf{Diffusion-generated time-marginals:}
We evolve an initial pmf $p_0$ by a semigroup $P_t$ on the cycle and study the resulting family
$p_t=p_0P_t$. On $\mathbb{Z}_m$, Fourier analysis yields fully explicit transition kernels, exact trigonometric
moments, and mixing rates to the uniform law.

\item \textbf{Drift-generated stationary laws:}
We construct nearest-neighbour generators $Q$ on the cycle that are reversible with respect to a prescribed
strictly positive pmf $\pi$. This yields a discrete analogue of a mean-reverting circular diffusion: the chain
has local (nearest-neighbour) dynamics on the circle and converges to a specified equilibrium preference
$\pi$. In particular, choosing $\pi$ to be discrete von Mises produces a natural discrete von Mises process in direct analogy with the continuous von Mises process of \cite{kent1978time}; likewise a
discrete wrapped Cauchy target yields an analogous wrapped-Cauchy equilibrium process.
\end{enumerate}
A key feature of working on the cycle graph is that these process constructions come with explicit and
computationally convenient consequences such as closed-form Fourier representations, exact trigonometric moments, and
explicit convergence rates. Since we are able to obtain explicit transition kernels, our results allow for likelihood based inference of discrete-circular time varying data, for practical problems listed above.  
\paragraph{Organization of the paper.}
Section 2 introduces diffusion semigroups on the cycle generated by fractional powers of the cycle Laplacian.
We derive an explicit Fourier series for the transition kernel, identify uniform stationarity, obtain exact
trigonometric moment formulas, and give convergence bounds to uniformity.
Section 3 develops the complementary construction: a reversible nearest-neighbour chain targeting an
arbitrary positive pmf $\pi$. Specializing $\pi$ yields discrete von Mises and discrete wrapped Cauchy processes,
and when the location parameter lies on the grid we derive closed-form normalizing constants and exact
trigonometric moments. 
\paragraph{Notation and conventions.}
Fix an integer $m\ge 3$. We write $\mathbb{Z}_m := \{0,1,\dots,m-1\}$ with arithmetic understood modulo $m$.
We identify $r\in\mathbb{Z}_m$ with the grid angle $\theta_r:=2\pi r/m\in[0,2\pi)$ and write
$D_m:=\{\theta_r:r\in\mathbb{Z}_m\}$.
We consider the cycle graph $G_m=(\mathbb{Z}_m,E)$ with edges $r\leftrightarrow r\pm 1$ (mod $m$), and its
(combinatorial) Laplacian $L$ acting on functions $f:\mathbb{Z}_m\to\mathbb{R}$ by
\[
(Lf)(r)=2f(r)-f(r+1)-f(r-1),
\]
with indices interpreted modulo $m$.

\section{Diffusion semigroups on the cycle}

See \cite{lovasz1993random} for a background on random walks on graphs. Fix $\alpha>0$ and $\beta\in(0,1]$. We define the semigroup (see \cite{Jacob2001}, Sec. 2.7 and for graph Laplacians see \cite{chung1997spectral}, Sec. 1.2 and Ch. 10),
\begin{equation}\label{eq:frac_semigroup}
P_t^{(\beta)} := \exp \big(-\alpha t L^\beta\big),\qquad t\ge 0,
\end{equation}
where $L^\beta$ is defined as follows:
since $L$ is real symmetric, there exists an orthonormal matrix $U$ and a diagonal matrix
$\Lambda=\mathrm{diag}(\lambda_0,\dots,\lambda_{m-1})$ with $\lambda_j\ge 0$ such that
\begin{equation}\label{eq:spectral_decomp}
L = U\Lambda U^\top .
\end{equation}
Since $L$ is positive semidefinite,
all eigenvalues satisfy $\lambda_j\ge 0$.
We then set
\begin{equation}\label{eq:Lbeta_def}
L^\beta := U\Lambda^\beta U^\top,\qquad
\Lambda^\beta:=\mathrm{diag}(\lambda_0^\beta,\dots,\lambda_{m-1}^\beta),
\end{equation}
and define
\begin{equation}\label{eq:semigroup_def}
P_t^{(\beta)} := \exp(-\alpha t L^\beta)
:= U \mathrm{diag} \big(e^{-\alpha t\lambda_0^\beta},\dots,e^{-\alpha t\lambda_{m-1}^\beta}\big) U^\top.
\end{equation}

The case $\beta=1$ corresponds to the standard heat semigroup on the graph (the continuous-time simple random walk). The case $\beta=\tfrac12$ is a discrete analogue of the Poisson semigroup on the continuous circle.

\paragraph{Markov interpretation.}
Because $L\mathbf{1}=0$, we have $P_t^{(\beta)}\mathbf{1}=\mathbf{1}$ for all $t$ (rows sum to $1$). Moreover, for $\beta\in(0,1)$ one can represent $P_t^{(\beta)}$ as a Bochner subordinate of the heat semigroup,
which in particular preserves positivity; see \cite[Sec. 4.3]{Jacob2001}. Thus $P_t^{(\beta)}(r,s)$ can be interpreted as transition probabilities of a continuous-time Markov chain $(X_t)_{t\ge 0}$ on $\mathbb{Z}_m$ with generator $-\alpha L^\beta$.

We also consider the associated angle-valued process
\[
\Theta_t := \theta_{X_t}\in D_m.
\]

\subsection{Transition kernel}

On the finite cyclic group $\mathbb{Z}_m$, define the characters
\[
\varphi_k(r):=\exp \Big(i\frac{2\pi k r}{m}\Big),\qquad k,r\in\mathbb{Z}_m.
\]
They satisfy the orthogonality relation
\begin{equation}\label{eq:orthogonality}
\frac{1}{m}\sum_{r=0}^{m-1}\varphi_k(r) \overline{\varphi_\ell(r)}
=\frac{1}{m}\sum_{r=0}^{m-1}e^{i\frac{2\pi (k-\ell)r}{m}}
=\mathbf{1}\{k=\ell\}.
\end{equation}

\begin{theorem}[Explicit transition kernel]\label{thm:kernel}
For each $k\in\mathbb{Z}_m$, $\varphi_k$ is an eigenfunction of $L$ with eigenvalue
\begin{equation}\label{eq:eigs}
\lambda_k = 2-2\cos \Big(\frac{2\pi k}{m}\Big)
= 4\sin^2 \Big(\frac{\pi k}{m}\Big).
\end{equation}
Consequently, for $\beta\in(0,1]$ and $t\ge 0$,
\begin{equation}\label{eq:kernelDFT}
P_t^{(\beta)}(r,s)
= \frac{1}{m}\sum_{k=0}^{m-1}\exp \big(-\alpha t \lambda_k^\beta\big) 
\exp \Big(i\frac{2\pi k}{m}(s-r)\Big).
\end{equation}
Moreover, $P_t^{(\beta)}(r,s)$ depends only on $s-r\pmod m$ (translation invariance).
\end{theorem}

Let $u$ denote the uniform pmf on $\mathbb{Z}_m$:
\[
u(r):=\frac1m,\qquad r\in\mathbb{Z}_m.
\]

\begin{cor}[Convergence to uniformity]\label{cor:uniform}
For all $\beta\in(0,1]$ and $t\ge 0$, $uP_t^{(\beta)}=u$. Moreover, for any initial pmf $p_0$ on $\mathbb{Z}_m$, the evolved pmf $p_t:=p_0P_t^{(\beta)}$ converges to $u$ as $t\to\infty$.
\end{cor}

\begin{proof}
The constant function $\mathbf{1}$ is the eigenfunction $\varphi_0$ and corresponds to eigenvalue $\lambda_0=0$. Therefore $P_t^{(\beta)}\mathbf{1}=\mathbf{1}$ and $u$ is stationary. Since $\lambda_k>0$ for $k\neq 0$, every non-constant Fourier mode is multiplied by $e^{-\alpha t\lambda_k^\beta}\to 0$, implying convergence to the uniform distribution.
\end{proof}

\subsection{Exact trigonometric moments}\label{sec:moments}

Let $p$ be a pmf on $\mathbb{Z}_m$. We use the discrete Fourier transform
\[
\widehat p(k) := \sum_{r=0}^{m-1} p(r)\exp \Big(-i\frac{2\pi k r}{m}\Big),
\qquad k\in\mathbb{Z}_m,
\]
and recall that $\theta_r=2\pi r/m$ and $\Theta_t=\theta_{X_t}$.

\begin{prop}\label{prop:moments}
Let $\beta\in(0,1]$ and let $p_t=p_0P_t^{(\beta)}$. Then for every $k\in\mathbb{Z}_m$,
\begin{equation}\label{eq:ft_evolution}
\widehat p_t(k)=\widehat p_0(k)\exp \big(-\alpha t \lambda_k^\beta\big).
\end{equation}
Equivalently, for any integer $\ell$ (only $\ell \bmod m$ matters),
\begin{equation}\label{eq:trigmom}
\mathbb{E} \left[e^{i \ell \Theta_t}\right]
= \mathbb{E} \left[e^{i \ell \Theta_0}\right]\exp \big(-\alpha t \lambda_{\ell \bmod m}^\beta\big).
\end{equation}
In particular, if $X_0=r_0$ (so $\Theta_0=\theta_{r_0}$), then
\[
\mathbb{E} \left[e^{i\ell(\Theta_t-\theta_{r_0})}\right]=\exp \big(-\alpha t \lambda_{\ell \bmod m}^\beta\big),
\quad
\mathbb{E} \left[\cos \big(\ell(\Theta_t-\theta_{r_0})\big)\right]=e^{-\alpha t\lambda_{\ell \bmod m}^\beta},
\quad
\mathbb{E} \left[\sin \big(\ell(\Theta_t-\theta_{r_0})\big)\right]=0.
\]
\end{prop}

\begin{cor}[A one-parameter concentration summary]\label{cor:momentmatch}
Consider the location family obtained by starting the diffusion from a point mass at $r_0$
(equivalently, by shifting the kernel). Then the mean direction is $\theta_{r_0}$ and the first resultant length equals
\[
R(t):=\left|\mathbb{E}[e^{i\Theta_t}]\right|
= \exp \big(-\alpha t \lambda_{1}^\beta\big).
\]
Thus moment-matching based on an empirical resultant length $\widehat R$ suggests
\[
\widehat{\alpha t} = -\frac{\log \widehat R}{\lambda_1^\beta}.
\]
\end{cor}

\subsection{Mixing rates}\label{sec:mixing}

Let $u$ denote the uniform pmf on $\mathbb{Z}_m$, $u(r)=1/m$. For a pmf $p$ we measure deviation from $u$
via the Radon-Nikodym derivative $p/u$:
\[
f(r):=\frac{p(r)}{u(r)}-1 = m p(r)-1,\qquad r\in\mathbb{Z}_m.
\]
For the time-marginal $p_t=p_0P_t^{(\beta)}$, define $f_t:=p_t/u-1$. Note that
\begin{equation}\label{eq:mean_zero_ft}
\sum_{r=0}^{m-1} u(r) f_t(r)  =  \sum_{r=0}^{m-1}\big(p_t(r)-u(r)\big) = 1-1 = 0.
\end{equation}

We use the weighted norms
\[
\|g\|_{p,u}^p := \sum_{r=0}^{m-1} u(r) |g(r)|^p,\qquad p\in[1,\infty),
\qquad
\|g\|_{\infty,u}:=\max_{0\le r\le m-1}|g(r)|.
\]
In particular,
\[
\|f_t\|_{2,u}^2=\frac{1}{m}\sum_{r=0}^{m-1}|f_t(r)|^2.
\]
Total variation admits the identity
\begin{equation}\label{eq:tv_ratio_identity}
\|p_t-u\|_{TV}
=\frac12\sum_{r=0}^{m-1}|p_t(r)-u(r)|
=\frac12\sum_{r=0}^{m-1}u(r) |f_t(r)|
=\frac12\|f_t\|_{1,u}.
\end{equation}

\begin{theorem}[Total-variation bound]\label{thm:mixing_uniform}
Let $\lambda_\star:=\min\{\lambda_k:k\neq 0\}=\lambda_1=4\sin^2(\pi/m)$. For any initial pmf $p_0$ and all $t\ge 0$,
\begin{equation}\label{eq:L2contract}
\|f_t\|_{2,u} \le e^{-\alpha t \lambda_\star^\beta} \|f_0\|_{2,u}.
\end{equation}
Consequently,
\begin{equation}\label{eq:TVbound}
\|p_t-u\|_{TV}\le \frac12\|f_t\|_{2,u}
\le \frac12 e^{-\alpha t \lambda_\star^\beta} \|f_0\|_{2,u}.
\end{equation}
In particular, if $p_0=\delta_{r_0}$, then $\|f_0\|_{2,u}=\sqrt{m-1}$ and
\[
\|p_t-u\|_{TV}\le \frac12\sqrt{m-1} e^{-\alpha t \lambda_\star^\beta}.
\]
\end{theorem}

\section{Nearest-neighbour chains with a prescribed stationary distribution}\label{sec:drift}

We now address the complementary construction problem: given a strictly positive target pmf
$\pi$ on $\mathbb{Z}_m$, construct a nearest-neighbour continuous-time Markov chain whose unique stationary
distribution is $\pi$. A convenient way to guarantee stationarity is to impose reversibility with respect to $\pi$. This can be regarded as the discrete analogue of the setup in \cite{kent1978time}.

\begin{prop}[A reversible nearest-neighbour construction]\label{prop:targetpi}
Let $\pi=(\pi_0,\dots,\pi_{m-1})$ be a strictly positive pmf on $\mathbb{Z}_m$ and let $\alpha>0$.
Define an infinitesimal generator $Q=(q_{r,s})_{r,s\in\mathbb{Z}_m}$ by
\begin{equation}\label{eq:rates_general}
q_{r,r+1}=\alpha\sqrt{\frac{\pi_{r+1}}{\pi_r}},\qquad
q_{r,r-1}=\alpha\sqrt{\frac{\pi_{r-1}}{\pi_r}},\qquad
q_{r,r}=-(q_{r,r+1}+q_{r,r-1}),
\end{equation}
and $q_{r,s}=0$ otherwise. Then:
\begin{enumerate}
\item[(i)] $Q$ is a valid generator of a nearest-neighbour continuous-time Markov chain on $\mathbb{Z}_m$
(all off-diagonal rates are nonnegative and rows sum to $0$);
\item[(ii)] the chain is reversible with respect to $\pi$, i.e.\ it satisfies detailed balance
\begin{equation}\label{eq:db}
\pi_r q_{r,s}=\pi_s q_{s,r}\qquad\text{for all }r,s\in\mathbb{Z}_m;
\end{equation}
\item[(iii)] consequently, $\pi$ is stationary: $\pi Q=0$ (equivalently, $\pi P_t=\pi$ for all $t\ge 0$).
\end{enumerate}
\end{prop}

When $\pi$ is uniform, $\pi_{r+1}/\pi_r=1$ and the rates $q_{r,r\pm 1}$ are constant, recovering the usual continuous-time
nearest-neighbour random walk on the cycle. For general $\pi$, the forward rate $q_{r,r+1}$ is larger when
$\pi_{r+1}>\pi_r$, biasing moves toward higher-probability states while maintaining reversibility.
\subsection{Discrete von Mises process}

Fix $\kappa\ge 0$ and $\mu\in[0,2\pi)$. The discrete von Mises pmf on the grid $D_m=\{\theta_r=2\pi r/m:r\in\mathbb{Z}_m\}$
is
\begin{equation}\label{eq:dvm_def}
\pi^{\mathrm{vM}}_r(\kappa,\mu)
:= \frac{\exp\{\kappa\cos(\theta_r-\mu)\}}{Z_m(\kappa,\mu)},
\qquad
Z_m(\kappa,\mu):=\sum_{j=0}^{m-1}\exp\{\kappa\cos(\theta_j-\mu)\}.
\end{equation}

\begin{cor}[von Mises stationary law]\label{cor:vm_stationary}
Applying Proposition \ref{prop:targetpi} with $\pi=\pi^{\mathrm{vM}}(\kappa,\mu)$ yields a reversible nearest-neighbour chain
on $\mathbb{Z}_m$ whose stationary distribution is $\pi^{\mathrm{vM}}(\kappa,\mu)$.
\end{cor}

If $\mu$ lies on the grid, i.e.\ $\mu=\theta_{r_0}$ for some $r_0\in\mathbb{Z}_m$, then $Z_m(\kappa,\mu)$ does not depend on $r_0$.
Indeed, replacing $r$ by $r-r_0$ permutes the summands in \eqref{eq:dvm_def}, so $Z_m(\kappa,\theta_{r_0})=Z_m(\kappa,\theta_0)$.
We therefore write $Z_m(\kappa):=Z_m(\kappa,\theta_0)$.

\begin{theorem}[Normalizing constant for discrete von Mises process]\label{thm:vm_norm}
Assume $\mu=\theta_{r_0}$ for some $r_0\in\mathbb{Z}_m$. Then
\begin{equation}\label{eq:vm_norm}
Z_m(\kappa)
= m\sum_{q\in\mathbb{Z}} I_{qm}(\kappa)
= m\Big(I_0(\kappa) + 2\sum_{q=1}^\infty I_{qm}(\kappa)\Big),
\end{equation}
where $I_n(\cdot)$ is the modified Bessel function of the first kind.
\end{theorem}

\begin{cor}[Exact trigonometric moments]\label{cor:vm_mom}
Assume $\mu=\theta_{r_0}$. Then for any integer $\ell$,
\begin{equation}\label{eq:vm_moment}
\mathbb{E}_{\pi^{\mathrm{vM}}} \left[e^{i \ell(\Theta-\mu)}\right]
= \frac{\sum_{q\in\mathbb{Z}} I_{\ell+qm}(\kappa)}{\sum_{q\in\mathbb{Z}} I_{qm}(\kappa)}.
\end{equation}
Equivalently,
\[
\mathbb{E}_{\pi^{\mathrm{vM}}} \left[e^{i\ell\Theta}\right]
=e^{i\ell\mu} \frac{\sum_{q\in\mathbb{Z}} I_{\ell+qm}(\kappa)}{\sum_{q\in\mathbb{Z}} I_{qm}(\kappa)}.
\]
\end{cor}

\subsection{Discrete wrapped Cauchy process}

Fix $\rho\in(0,1)$ and $\mu\in[0,2\pi)$. Consider the Poisson kernel values on the grid
\begin{equation}\label{eq:poisson_kernel_grid}
w_r(\rho,\mu):=\frac{1-\rho^2}{1-2\rho\cos(\theta_r-\mu)+\rho^2},\qquad r\in\mathbb{Z}_m,
\end{equation}
and define the discrete wrapped Cauchy pmf by normalization,
\[
\pi^{\mathrm{WC}}_r(\rho,\mu) := \frac{w_r(\rho,\mu)}{\sum_{j=0}^{m-1} w_j(\rho,\mu)}.
\]

\begin{cor}[Wrapped Cauchy stationary law]\label{cor:wc_stationary}
Applying Proposition \ref{prop:targetpi} with $\pi=\pi^{\mathrm{WC}}(\rho,\mu)$ yields a reversible nearest-neighbour chain
on $\mathbb{Z}_m$ whose stationary distribution is $\pi^{\mathrm{WC}}(\rho,\mu)$.
\end{cor}

If $\mu=\theta_{r_0}$ lies on the grid, then $\{ \theta_r-\mu:  r\in\mathbb{Z}_m \}$ is just a permutation of $\{\theta_r: r\in\mathbb{Z}_m\}$,
so the normalizer $\sum_r w_r(\rho,\mu)$ does not depend on $r_0$. We henceforth assume $\mu=\theta_{r_0}$.

\begin{theorem}[Normalizing constant and moments]\label{thm:wc_norm_mom}
Assume $\mu=\theta_{r_0}$ for some $r_0\in\mathbb{Z}_m$. Then
\begin{equation}\label{eq:wc_norm}
\sum_{r=0}^{m-1} w_r(\rho,\mu) = m \frac{1+\rho^m}{1-\rho^m},
\end{equation}
and therefore
\begin{equation}\label{eq:wc_pmf}
\pi^{\mathrm{WC}}_r(\rho,\mu)
=\frac{1-\rho^m}{m(1+\rho^m)} \frac{1-\rho^2}{1-2\rho\cos(\theta_r-\mu)+\rho^2}.
\end{equation}
Moreover, for $\ell\in\{0,1,\dots,m-1\}$,
\begin{equation}\label{eq:wc_moment}
\mathbb{E}_{\pi^{\mathrm{WC}}} \left[e^{i \ell(\Theta-\mu)}\right]
=\frac{\rho^\ell+\rho^{m-\ell}}{1+\rho^m}.
\end{equation}
\end{theorem}

\section*{Acknowledgements}

The author would like to thank Prof. Karthik Sriram for some helpful discussions on \cite{mardia2023families} and for feedback on an earlier version of the paper.
\printbibliography

\section*{Appendix}
\subsection*{Proof of Theorem \ref{thm:kernel}}
\begin{proof}
Recall $(Lf)(r)=2f(r)-f(r+1)-f(r-1)$ with indices modulo $m$.
For $\varphi_k(r)=e^{i2\pi k r/m}$,
\[
\varphi_k(r+1)=e^{i\frac{2\pi k(r+1)}{m}}
=e^{i\frac{2\pi k}{m}}\varphi_k(r),
\qquad
\varphi_k(r-1)=e^{i\frac{2\pi k(r-1)}{m}}
=e^{-i\frac{2\pi k}{m}}\varphi_k(r).
\]
Therefore
\[
(L\varphi_k)(r)
= \Big(2-e^{i\frac{2\pi k}{m}}-e^{-i\frac{2\pi k}{m}}\Big)\varphi_k(r)
= \big(2-2\cos(2\pi k/m)\big)\varphi_k(r),
\]
so $\varphi_k$ is an eigenfunction with eigenvalue $\lambda_k=2-2\cos(2\pi k/m)$.
Using $1-\cos x = 2\sin^2(x/2)$ gives $\lambda_k=4\sin^2(\pi k/m)$. These follow from the fact that $L$ is a circulant matrix, see \cite{chung1997spectral} for details.

Define the inner product $\langle f,g\rangle := \frac{1}{m}\sum_{r=0}^{m-1} f(r)\overline{g(r)}$.
By \eqref{eq:orthogonality}, $\{\varphi_k\}_{k=0}^{m-1}$ is an orthonormal basis of $\mathbb{C}^m$.
Hence any function $f:\mathbb{Z}_m\to\mathbb{C}$ admits the Fourier expansion
\begin{equation}\label{eq:fourier_expansion}
f(r)=\sum_{k=0}^{m-1}\widehat f(k) \varphi_k(r),
\qquad
\widehat f(k):=\langle f,\varphi_k\rangle=\frac{1}{m}\sum_{r=0}^{m-1} f(r)\overline{\varphi_k(r)}.
\end{equation}
Since $L\varphi_k=\lambda_k\varphi_k$, linearity gives
\[
Lf=\sum_{k=0}^{m-1}\widehat f(k) \lambda_k \varphi_k,
\qquad
L^\beta f=\sum_{k=0}^{m-1}\widehat f(k) \lambda_k^\beta \varphi_k,
\]

Let $P_t^{(\beta)}=\exp(-\alpha t L^\beta)$. Using the power-series definition of the matrix exponential,
\[
P_t^{(\beta)}f
=\sum_{n=0}^\infty \frac{(-\alpha t)^n}{n!}(L^\beta)^n f.
\]
Since $(L^\beta)^n \varphi_k = (\lambda_k^\beta)^n \varphi_k$, we get
\[
P_t^{(\beta)}\varphi_k
=\sum_{n=0}^\infty \frac{(-\alpha t)^n}{n!}(\lambda_k^\beta)^n\varphi_k
=e^{-\alpha t\lambda_k^\beta} \varphi_k.
\]
Therefore, for general $f$ with expansion \eqref{eq:fourier_expansion},
\[
(P_t^{(\beta)}f)(r)=\sum_{k=0}^{m-1}\widehat f(k) e^{-\alpha t\lambda_k^\beta} \varphi_k(r).
\]

Apply the previous formula to the delta function $\delta_s(\cdot):=\mathbf{1}\{\cdot=s\}$.
Its Fourier coefficients are
\[
\widehat{\delta_s}(k)
=\frac{1}{m}\sum_{r=0}^{m-1}\delta_s(r)\overline{\varphi_k(r)}
=\frac{1}{m}\overline{\varphi_k(s)}
=\frac{1}{m}e^{-i\frac{2\pi k s}{m}}.
\]
Thus
\[
P_t^{(\beta)}(r,s)=(P_t^{(\beta)}\delta_s)(r)
=\sum_{k=0}^{m-1}\widehat{\delta_s}(k) e^{-\alpha t\lambda_k^\beta} \varphi_k(r)
=\frac{1}{m}\sum_{k=0}^{m-1}e^{-\alpha t\lambda_k^\beta} 
e^{i\frac{2\pi k}{m}(r-s)}.
\]
Replacing $r-s$ by $s-r$ (equivalently taking complex conjugates; the result is real) yields
\eqref{eq:kernelDFT}.

The final expression depends on $r$ and $s$ only through $r-s\pmod m$, hence
$P_t^{(\beta)}(r,s)=\kappa_t^{(\beta)}(s-r)$ for some function $\kappa_t^{(\beta)}$ on $\mathbb{Z}_m$.
\end{proof}

\subsection*{Proof of Proposition \ref{prop:moments}}
\begin{proof}
Fix $k\in\mathbb{Z}_m$ and consider the complex-valued function $f_k(r):=\exp(-i2\pi k r/m)$.
By Theorem \ref{thm:kernel}, $f_k$ is an eigenfunction of $L$ with eigenvalue $\lambda_k$, hence it is also
an eigenfunction of $L^\beta$ with eigenvalue $\lambda_k^\beta$, and therefore of $P_t^{(\beta)}=\exp(-\alpha tL^\beta)$ with eigenvalue
$e^{-\alpha t\lambda_k^\beta}$. Concretely, for every $r\in\mathbb{Z}_m$,
\begin{equation}\label{eq:eig_action}
(P_t^{(\beta)}f_k)(r)=e^{-\alpha t\lambda_k^\beta} f_k(r).
\end{equation}

Now use the Markov property in the form of the tower rule. Since $p_t=p_0P_t^{(\beta)}$,
\[
\widehat p_t(k)=\sum_{r=0}^{m-1} p_t(r) f_k(r)
=\sum_{r=0}^{m-1}\Big(\sum_{s=0}^{m-1}p_0(s)P_t^{(\beta)}(s,r)\Big)f_k(r).
\]
Swap the finite sums to obtain
\[
\widehat p_t(k)=\sum_{s=0}^{m-1}p_0(s)\sum_{r=0}^{m-1}P_t^{(\beta)}(s,r) f_k(r)
=\sum_{s=0}^{m-1}p_0(s) (P_t^{(\beta)}f_k)(s).
\]
Applying \eqref{eq:eig_action} gives
\[
\widehat p_t(k)=\sum_{s=0}^{m-1}p_0(s) e^{-\alpha t\lambda_k^\beta}f_k(s)
=e^{-\alpha t\lambda_k^\beta}\sum_{s=0}^{m-1}p_0(s)e^{-i2\pi k s/m}
=e^{-\alpha t\lambda_k^\beta} \widehat p_0(k),
\]
which proves \eqref{eq:ft_evolution}.

For the moment identity, note that $e^{i\ell\Theta_t}=e^{i\ell\theta_{X_t}}=\exp(i2\pi \ell X_t/m)$, so
\[
\mathbb{E} \left[e^{i\ell\Theta_t}\right]
=\sum_{r=0}^{m-1} p_t(r)\exp \Big(i\frac{2\pi \ell r}{m}\Big)
=\sum_{r=0}^{m-1} p_t(r)\exp \Big(-i\frac{2\pi (m-\ell) r}{m}\Big)
=\widehat p_t(m-\ell \bmod m).
\]
Applying \eqref{eq:ft_evolution} with $k=m-\ell\bmod m$, and using $\lambda_{m-\ell}=\lambda_\ell$, yields
\[
\mathbb{E} \left[e^{i\ell\Theta_t}\right]
=\widehat p_0(m-\ell) e^{-\alpha t\lambda_\ell^\beta}
=\mathbb{E} \left[e^{i\ell\Theta_0}\right]e^{-\alpha t\lambda_{\ell\bmod m}^\beta},
\]
which is \eqref{eq:trigmom}.

Finally, if $X_0=r_0$ then $\Theta_0=\theta_{r_0}$ and
\[
\mathbb{E} \left[e^{i\ell(\Theta_t-\theta_{r_0})}\right]
=e^{-i\ell\theta_{r_0}}\mathbb{E} \left[e^{i\ell\Theta_t}\right]
=e^{-i\ell\theta_{r_0}}\cdot e^{i\ell\theta_{r_0}}e^{-\alpha t\lambda_{\ell\bmod m}^\beta}
=e^{-\alpha t\lambda_{\ell\bmod m}^\beta}.
\]
Taking real and imaginary parts gives the cosine and sine statements.
\end{proof}

\subsection*{Proof of Theorem \ref{thm:mixing_uniform}}
\textbf{Notation.}
For a function $g:\mathbb{Z}_m\to\mathbb{C}$, with $\widehat g$, its discrete Fourier transform as defined above. The inversion formula is
\begin{equation}\label{eq:DFT_inverse}
g(r)=\frac{1}{m}\sum_{k=0}^{m-1}\widehat g(k)\exp \Big(i\frac{2\pi k r}{m}\Big),
\end{equation}
and Parseval's identity
\begin{equation}\label{eq:parseval}
\|g\|_{2,u}^2=\frac{1}{m}\sum_{r=0}^{m-1}|g(r)|^2=\frac{1}{m^2}\sum_{k=0}^{m-1}|\widehat g(k)|^2.
\end{equation}
See \cite{stein2011fourier} for a reference.
\begin{proof}
Since $u$ is stationary for $P_t^{(\beta)}$, we have $uP_t^{(\beta)}=u$. Therefore
\[
p_t-u=(p_0-u)P_t^{(\beta)}.
\]
Multiplying by $m$ and using $f_t=m(p_t-u)$ gives the linear evolution
\begin{equation}\label{eq:ft_evolves}
f_t=f_0P_t^{(\beta)}.
\end{equation}

From Proposition \ref{prop:moments},
each Fourier coefficient evolves as
\begin{equation}\label{eq:ft_fourier_evolution}
\widehat f_t(k)=\widehat f_0(k) e^{-\alpha t\lambda_k^\beta},\qquad k\in\mathbb{Z}_m.
\end{equation}
Moreover, the mean-zero property \eqref{eq:mean_zero_ft} implies
\begin{equation}\label{eq:k0_is_zero}
\widehat f_t(0)=\sum_{r=0}^{m-1} f_t(r)= m\sum_{r=0}^{m-1}u(r)f_t(r)=0,
\end{equation}
so only modes $k\neq 0$ contribute to $\|f_t\|_{2,u}$.

Using Parseval \eqref{eq:parseval} and \eqref{eq:ft_fourier_evolution},
\[
\|f_t\|_{2,u}^2
=\frac{1}{m^2}\sum_{k=0}^{m-1}|\widehat f_t(k)|^2
=\frac{1}{m^2}\sum_{k\neq 0}|\widehat f_0(k)|^2 e^{-2\alpha t\lambda_k^\beta}.
\]
Since $\lambda_k\ge \lambda_\star$ for all $k\neq 0$, we obtain
\[
\|f_t\|_{2,u}^2
\le e^{-2\alpha t\lambda_\star^\beta}\frac{1}{m^2}\sum_{k\neq 0}|\widehat f_0(k)|^2
\le e^{-2\alpha t\lambda_\star^\beta}\frac{1}{m^2}\sum_{k=0}^{m-1}|\widehat f_0(k)|^2
= e^{-2\alpha t\lambda_\star^\beta}\|f_0\|_{2,u}^2,
\]
and taking square roots yields \eqref{eq:L2contract}.

By \eqref{eq:tv_ratio_identity} and Cauchy--Schwarz under the probability measure $u$,
\[
\|f_t\|_{1,u}=\sum_{r=0}^{m-1}u(r)|f_t(r)|
\le \Big(\sum_{r=0}^{m-1}u(r)\Big)^{1/2}\Big(\sum_{r=0}^{m-1}u(r)|f_t(r)|^2\Big)^{1/2}
=\|f_t\|_{2,u}.
\]
Therefore $\|p_t-u\|_{TV}=\tfrac12\|f_t\|_{1,u}\le \tfrac12\|f_t\|_{2,u}$, and \eqref{eq:TVbound} follows from \eqref{eq:L2contract}.

If $p_0=\delta_{r_0}$, then $f_0(r_0)=m-1$ and $f_0(r)=-1$ for $r\neq r_0$. Hence
\[
\|f_0\|_{2,u}^2=\frac{1}{m}\Big((m-1)^2+(m-1)\cdot 1\Big)=m-1,
\]
so $\|f_0\|_{2,u}=\sqrt{m-1}$.
\end{proof}

\subsection*{Proof of Proposition \ref{prop:targetpi}}
\begin{proof}
\emph{(i)} Since $\pi_r>0$ for all $r$, the off-diagonal rates $q_{r,r\pm 1}$ in \eqref{eq:rates_general} are well-defined and strictly
positive. By definition $q_{r,s}=0$ for non-neighbours, and $q_{r,r}=-(q_{r,r+1}+q_{r,r-1})$, hence
\[
\sum_{s\in\mathbb{Z}_m} q_{r,s}=q_{r,r+1}+q_{r,r-1}+q_{r,r}=0,
\]
so each row sums to $0$ and $Q$ is a valid generator.

\emph{(ii)} If $s=r+1$, then using \eqref{eq:rates_general},
\[
\pi_r q_{r,r+1}
=\pi_r\alpha\sqrt{\frac{\pi_{r+1}}{\pi_r}}
=\alpha\sqrt{\pi_r\pi_{r+1}}
=\pi_{r+1}\alpha\sqrt{\frac{\pi_r}{\pi_{r+1}}}
=\pi_{r+1}q_{r+1,r}.
\]
The same computation holds for $s=r-1$. For all other $s$ we have $q_{r,s}=q_{s,r}=0$.
Therefore detailed balance \eqref{eq:db} holds for all pairs $(r,s)$, and the chain is reversible with respect to $\pi$.

\emph{(iii)} Stationarity follows by summing the detailed-balance equalities over $r$ for each fixed $s$:
\[
(\pi Q)_s=\sum_{r\in\mathbb{Z}_m}\pi_r q_{r,s}
=\sum_{r\in\mathbb{Z}_m}\pi_s q_{s,r}
=\pi_s\sum_{r\in\mathbb{Z}_m}q_{s,r}
=\pi_s\cdot 0=0,
\]
where we used the row-sum property from (i) in the last step. Hence $\pi Q=0$.
\end{proof}

\subsection*{Proof of Theorem \ref{thm:vm_norm}}

\begin{proof}
We use the standard Fourier--Bessel expansion (valid for all $\kappa\ge 0$ and $\theta\in\mathbb{R}$)
\begin{equation}\label{eq:bessel_expansion}
e^{\kappa\cos\theta}=\sum_{n\in\mathbb{Z}} I_n(\kappa) e^{in\theta}.
\end{equation}
With $\mu=\theta_{r_0}$, write $\theta_r-\mu = 2\pi(r-r_0)/m$. Then
\[
Z_m(\kappa,\mu)=\sum_{r=0}^{m-1} e^{\kappa\cos(\theta_r-\mu)}
=\sum_{r=0}^{m-1}\sum_{n\in\mathbb{Z}} I_n(\kappa)e^{in(\theta_r-\mu)}.
\]
Interchange the finite sum over $r$ with the absolutely convergent series in $n$ to obtain
\[
Z_m(\kappa,\mu)=\sum_{n\in\mathbb{Z}} I_n(\kappa)e^{-in\mu}\sum_{r=0}^{m-1} e^{in\theta_r}.
\]
The inner sum is the root-of-unity filter:
\begin{equation}\label{eq:root_filter}
\sum_{r=0}^{m-1} e^{in\theta_r}=\sum_{r=0}^{m-1} e^{i2\pi nr/m}
=
\begin{cases}
m,& m\mid n,\\
0,& m\nmid n.
\end{cases}
\end{equation}
Hence only indices $n=qm$ survive, giving
\[
Z_m(\kappa,\mu)=m\sum_{q\in\mathbb{Z}} I_{qm}(\kappa)e^{-iqm\mu}.
\]
Finally, since $\mu=\theta_{r_0}=2\pi r_0/m$, we have $e^{-iqm\mu}=e^{-i2\pi q r_0}=1$, so
$Z_m(\kappa,\mu)=m\sum_{q\in\mathbb{Z}} I_{qm}(\kappa)=Z_m(\kappa)$, which is \eqref{eq:vm_norm}.
The second equality follows from $I_{-n}(\kappa)=I_n(\kappa)$.
\end{proof}

\subsection*{Proof of Corollary \ref{cor:vm_mom}}
\begin{proof}
By definition,
\[
\mathbb{E}_{\pi^{\mathrm{vM}}} \left[e^{i\ell(\Theta-\mu)}\right]
=\frac{1}{Z_m(\kappa,\mu)}\sum_{r=0}^{m-1} e^{\kappa\cos(\theta_r-\mu)}e^{i\ell(\theta_r-\mu)}.
\]
Expand $e^{\kappa\cos(\theta_r-\mu)}$ using \eqref{eq:bessel_expansion}:
\[
\sum_{r=0}^{m-1} e^{\kappa\cos(\theta_r-\mu)}e^{i\ell(\theta_r-\mu)}
=\sum_{r=0}^{m-1}\sum_{n\in\mathbb{Z}} I_n(\kappa) e^{i(n+\ell)(\theta_r-\mu)}.
\]
Interchange sums and apply the root-of-unity filter \eqref{eq:root_filter} to the inner sum in $r$:
\[
\sum_{r=0}^{m-1} e^{i(n+\ell)(\theta_r-\mu)}
=e^{-i(n+\ell)\mu}\sum_{r=0}^{m-1}e^{i(n+\ell)\theta_r}
=
\begin{cases}
m e^{-i(n+\ell)\mu}, & m\mid(n+\ell),\\
0,& m\nmid(n+\ell).
\end{cases}
\]
Thus only indices $n=-\ell+qm$ contribute, and the numerator becomes
\[
m\sum_{q\in\mathbb{Z}} I_{-\ell+qm}(\kappa) e^{-i(qm)\mu}.
\]
When $\mu=\theta_{r_0}$, $e^{-iqm\mu}=1$. Using $I_{-n}=I_n$ and reindexing $q\mapsto -q$ gives
\[
m\sum_{q\in\mathbb{Z}} I_{-\ell+qm}(\kappa)
= m\sum_{q\in\mathbb{Z}} I_{\ell+qm}(\kappa).
\]
Divide by $Z_m(\kappa,\mu)=Z_m(\kappa)=m\sum_{q\in\mathbb{Z}}I_{qm}(\kappa)$ from Theorem \ref{thm:vm_norm}
to obtain \eqref{eq:vm_moment}. The second formula follows from
$e^{i\ell\Theta}=e^{i\ell\mu}e^{i\ell(\Theta-\mu)}$.
\end{proof}

\subsection*{Proof of Theorem \ref{thm:wc_norm_mom}}

\begin{proof}
	We use the classical Poisson kernel Fourier series (\cite{stein2011fourier}, Ch. 3), valid for $|\rho|<1$:
\begin{equation}\label{eq:poisson_fourier}
\frac{1-\rho^2}{1-2\rho\cos\theta+\rho^2}
=\sum_{n\in\mathbb{Z}}\rho^{|n|}e^{in\theta}
=1+2\sum_{n=1}^\infty \rho^n\cos(n\theta).
\end{equation}

With $\mu=\theta_{r_0}$, set $\theta=\theta_r-\mu$. Then by \eqref{eq:poisson_fourier},
\[
w_r(\rho,\mu)=\sum_{n\in\mathbb{Z}}\rho^{|n|}e^{in(\theta_r-\mu)}.
\]
Summing over $r$ and interchanging the (absolutely convergent) series with the finite sum gives
\[
\sum_{r=0}^{m-1} w_r(\rho,\mu)
=\sum_{n\in\mathbb{Z}}\rho^{|n|}e^{-in\mu}\sum_{r=0}^{m-1}e^{in\theta_r}.
\]
The root-of-unity filter yields
\begin{equation}\label{eq:root_filter_wc}
\sum_{r=0}^{m-1}e^{in\theta_r}
=\sum_{r=0}^{m-1}e^{i2\pi nr/m}
=
\begin{cases}
m,& m\mid n,\\
0,& m\nmid n.
\end{cases}
\end{equation}
Hence only $n=qm$ survive:
\[
\sum_{r=0}^{m-1} w_r(\rho,\mu)
=m\sum_{q\in\mathbb{Z}}\rho^{|qm|}e^{-iqm\mu}.
\]
Since $\mu=\theta_{r_0}=2\pi r_0/m$, we have $e^{-iqm\mu}=e^{-i2\pi qr_0}=1$, so
\[
\sum_{r=0}^{m-1} w_r(\rho,\mu)=m\sum_{q\in\mathbb{Z}}\rho^{|qm|}
=m\Big(1+2\sum_{q=1}^\infty \rho^{qm}\Big)
=m \frac{1+\rho^m}{1-\rho^m},
\]
which is \eqref{eq:wc_norm}. Dividing $w_r$ by this normalizer gives \eqref{eq:wc_pmf}.

By definition,
\[
\mathbb{E}_{\pi^{\mathrm{WC}}} \left[e^{i\ell(\Theta-\mu)}\right]
=\frac{\sum_{r=0}^{m-1} w_r(\rho,\mu) e^{i\ell(\theta_r-\mu)}}{\sum_{r=0}^{m-1} w_r(\rho,\mu)}.
\]
Use \eqref{eq:poisson_fourier} again:
\[
\sum_{r=0}^{m-1} w_r(\rho,\mu) e^{i\ell(\theta_r-\mu)}
=\sum_{r=0}^{m-1}\sum_{n\in\mathbb{Z}}\rho^{|n|}e^{i(n+\ell)(\theta_r-\mu)}.
\]
Interchange sums and apply the filter \eqref{eq:root_filter_wc} to $\sum_r e^{i(n+\ell)\theta_r}$:
\[
\sum_{r=0}^{m-1} e^{i(n+\ell)(\theta_r-\mu)}
=e^{-i(n+\ell)\mu}\sum_{r=0}^{m-1}e^{i(n+\ell)\theta_r}
=
\begin{cases}
m e^{-i(n+\ell)\mu},& m\mid(n+\ell),\\
0,& m\nmid(n+\ell).
\end{cases}
\]
Thus only indices $n=-\ell+qm$ contribute, and the numerator becomes
\[
m\sum_{q\in\mathbb{Z}}\rho^{|qm-\ell|}e^{-iqm\mu}.
\]
As before, $e^{-iqm\mu}=1$ when $\mu=\theta_{r_0}$. For $\ell\in\{0,1,\dots,m-1\}$ we compute the sum explicitly:
\[
\sum_{q\in\mathbb{Z}}\rho^{|qm-\ell|}
=\rho^\ell+\sum_{q=1}^\infty \rho^{qm-\ell}+\sum_{q=1}^\infty \rho^{qm+\ell}
=\rho^\ell+\frac{\rho^{m-\ell}}{1-\rho^m}+\frac{\rho^{m+\ell}}{1-\rho^m}
=\frac{\rho^\ell+\rho^{m-\ell}}{1-\rho^m}.
\]
Therefore the numerator equals $m(\rho^\ell+\rho^{m-\ell})/(1-\rho^m)$.
Dividing by the normalizer $m(1+\rho^m)/(1-\rho^m)$ from \eqref{eq:wc_norm} yields \eqref{eq:wc_moment}.
\end{proof}

\end{document}